[1994/12/01]
\documentclass[reqno,12pt]{amsart}
\usepackage{bbm}
\def\1{\mathbbm{1}}

\usepackage{amsmath,amsthm}
\usepackage[mathscr]{euscript}
\usepackage{mathrsfs}
\usepackage{chngcntr}
\usepackage{amsfonts}
\usepackage{amssymb}
\usepackage{latexsym}
\usepackage{tikz}
\usepackage{graphics}

\newtheorem{theorem}{Theorem}[section]
\newtheorem{proposition}[theorem]{Proposition}
\newtheorem{lemma}[theorem]{Lemma}

\newtheorem{example}[theorem]{Example}
\newtheorem{definition}[theorem]{Definition}
\newtheorem{remark}[theorem]{Remark}

\numberwithin{equation}{section}

\newcommand{\T}{\mathbb {T}}
\newcommand{\A}{\mathbb {A}}

\newcommand{\C}{\mathbb {C}}
\newcommand{\R}{\mathbb {R}}
\newcommand{\F}{\mathbb {F}}
\newcommand{\D}{\mathbb {D}}
\newcommand{\Pp}{\mathbb {P}}
\newcommand{\Ss}{\mathbb {S}}

\newcommand{\calP}{\mathcal {P}}
\newcommand{\calL}{\mathcal {L}}

\newcommand{\calR}{\mathcal {R}}

\newcommand{\al}{\alpha}
\newcommand{\ga}{\gamma}
\newcommand{\thh}{\theta}
\newcommand{\la}{\lambda}
\newcommand{\si}{\sigma}

\newcommand{\Om}{\Omega}
\newcommand{\om}{\omega}

\title{On the maximal regularity for a classe of Volterra integro-differential equations}
\author{A. AMANSAG, H. BOUNIT, A. DRIOUICH, and S. HADD}
\address{Department of Mathematics, Faculty of Sciences, Ibn Zohr University, Hay Dakhla, BP8106, 80000--Agadir, Morocco; ahmed.amansag@uiz.ac.ma, h.bounit@uiz.ac.ma, a.driouich@uiz.ac.ma, s.hadd@uiz.ac.ma}
\thanks{02/10/2020 Version}
\subjclass[2010]{***(primary), and ***(secondary)}
 \keywords{Operator semigroup, Boundary perturbation, Maximal regularity, Banach space, admissible control, admissible observation, Integro-differential equation}

\begin{document}
\maketitle

\renewcommand{\sectionmark}[1]{}
\begin{abstract}
We propose an approach  based on perturbation theory to establish maximal $L^p$-regularity for a class of Volterra integro-differential equations. As the left shift semigroup is involved for such equations, we study maximal regularity on Bergman spaces for autonomous and non-autonomous integro-differential equations. Our method is based on the formulation of the integro-differential equations to a Cauchy problems, infinite dimensional systems theory and some recent results on the perturbation of maximal $L^p$-regularity (see \cite{AmBoDrHa1,AmBoDrHa2}). Applications to heat equations driven by the Dirichlet (or Neumann)-Laplacian are considered.
 \end{abstract}

\section{Introduction}
In recent years, somewhat more progress on the concept of maximal $L^p$-regularity ($p\in (1,\infty)$)  has been made in the evolution equations literature. This property plays an important role in the well-posedness of  nonlinear evolution equations, quasilinear ones and non-autonomous evolution ones. Various approaches have been proposed for the concept of maximal $L^p$-regularity, we cite the variational approach e.g. \cite{Lions61}, the operator one e.g. \cite{ArRaFoPo}, \cite{KunstmannWeis}, and the perturbation one e.g. \cite{AmBoDrHa1, AmBoDrHa2}. For more facts on this property, the reader is invited to consult this non-exhaustive list \cite{Dore}, \cite{DenHiePru}, \cite{KunWei}, \cite{KunstmannWeis}, \cite{ArRaFoPo}, \cite{HieMon}, \cite{LutzWeis}, \cite{LutzWeis2} and references therein.

This paper focuses on proving the maximal $L^p$-regularity for some classes of Volterra integro-differential equations using the recent results based on the perturbation approach developed in \cite{AmBoDrHa1, AmBoDrHa2}. On the one hand, the results displayed throughout this article draw from our recent papers \cite{AmBoDrHa1,AmBoDrHa2} where the above equations are studied with kernels $a(\cdot)= 0$, and, on the other hand, from B\'arta \cite{Barta1} where these problems are studied on UMD spaces by using the concept of $R$-sectoriality. One remarkable fact is
that in the context of UMD spaces the left shift semigroup  on Bergman space enjoys the maximal $L^p$-regularity. \\

We first consider in Section \ref{sec:autonomous volterra witout BC} the following autonomous Volterra integro-differential equation
\begin{align}
\label{IACP}
\begin{cases}
\dot{z}(t) = \A z(t)+\displaystyle\int_0^t a(t-s)Fz(s)ds +f(t), & t\ge 0,  \\
z(0) = 0,
\end{cases}
\end{align}
where $A:D(A)\subset X\to X$ is the generator of a $C_0$-semigroup $\T:=(\T(t))_{t\ge 0}$ on a Banach space $X$ and $F:D(A)\to X$ a linear operator, and $a:\C\to \C$ and $f:[0,\infty)\to X$ are measurable functions.

In a suitable product space, the previous problem is reformulated as the following non-homogeneous problem
\begin{align*}
\begin{cases}
\dot{\varrho}(t) = \mathfrak{A} \varrho(t) + \zeta(t), & t\ge 0, \\
\varrho(0) = (\begin{smallmatrix} 0\\ 0\end{smallmatrix}),
\
\end{cases}
\end{align*}
where $\mathfrak{A}$ is a matrix operator (see Section \ref{sec:autonomous volterra witout BC}). Using a recent perturbation result of maximal $L^p$-regularity, we prove, under assumptions, that the operator $\mathfrak{A}$ has maximal $L^p$-regularity and we give an estimate for the solution of the problem \eqref{IACP}.

In \cite{Pruss}, the author studied maximal $L^p$-regularity of type $C^\alpha$ of \eqref{IACP}, which differs from the maximal $L^p$-regularity of type $L^p$ presented in this paper. Here we use a direct approach in the treatment of \eqref{IACP} without appealing the concept of $\kappa$-regular kernels as in the paper \cite{Zacher}.

In Section \ref{sec:volterra with BC}, we study the maximal $L^p$-regularity for Volterra integro-differential equations with boundary conditions of the form
\begin{align}\label{InDiffwBC}
\begin{cases}
\dot{z}(t)=A_m z(t)+\int_0^t a(t-s) P z(s) ds + f(t),& t\in [0,T],\cr
z(0)=0,\cr Gz(t)=Kz(t),& t\in [0,T],
\end{cases}
\end{align}
where $A_m:Z\to X$ is a closed linear operator with $Z$ is a Banach space that is densely and continuously embedded in the Banach space $X$. $G, K: Z\to U$ are linear operators with $U$ another Banach space. $P:Z\to X$ is a linear operator.

In \cite{AmBoDrHa1} and \cite{AmBoDrHa2}, the authors studied the problem \eqref{InDiffwBC} in the case $a(\cdot) = 0$ using the feedback theory of infinite dimensional linear systems.

In order to prove maximal $L^p$-regularity of \eqref{InDiffwBC}, we reformulate the problem as

\begin{equation*}
\begin{cases}
\dot{\rho}(t) = \mathfrak{A}\rho(t)+\zeta(t) , & \mbox{ } t\in [0, T] \\
\rho(0)=0, & \mbox{}
\
\end{cases}
\end{equation*}
where $\mathfrak{A}$ is some matrix operator(see Section \ref{sec:volterra with BC}). Using results from \cite{AmBoDrHa2}, we prove, under some assumptions, that $\mathfrak{A}$ has maximal $L^p$-regularity and we derive a useful estimate satisfied by the solution of \eqref{InDiffwBC}.

Section \ref{sec:Max Reg review} is devoted to recall the definition of maximal $L^p$-regularity and a useful perturbation result. In Section \ref{sec:autonomous volterra witout BC}, we study well-posedness and maximal $L^p$-regularity of evolution equations \eqref{IACP}. In Section \ref{sec:volterra with BC}, we review some useful results on feedback theory of infinite dimensional systems and prove maximal $L^p$-regularity of \eqref{InDiffwBC} under suitable assumptions.

\section{Background on maximal $L^p$-regularity for Cauchy problems}
\label{sec:Max Reg review}
In this section, we collect necessary background on maximal $L^p$-regularity that will be used in this paper. Let $\mathscr{X}$ be a Banach space with norm $\Vert \cdot \Vert,$ $p\in (1,\infty)$ a real number, and $\mathscr{A}:D(\mathscr{A})\subset \mathscr{X}\to \mathscr{X}$  a closed linear operator.
\begin{definition}
	We say that $\mathscr{A}$ has the maximal $L^p$-regularity, and we write
	$\mathscr{A}\in MR_p(0, T; \mathscr{X})$, if for every $f\in L^p([0,T],\mathscr{X})$ there exists a unique  $u\in W^{1,p}([0,T],\mathscr{X})\cap L^p([0,T],D(\mathscr{A}))$ such that
	\begin{equation*}
	\dot{u} (t)  = \mathscr{A}u(t) + f(t) \quad \text{ for } t\in [0,T] \quad \text{ and } u(0) = 0.
	\end{equation*}
\end{definition}
By "maximal" we mean that the applications $f$, $\mathscr{A} u$ and $\dot{u}$ have the same regularity. According to  the closed graph theorem,  if  $\mathscr{A}$ has maximal $L^p$-regularity,
\begin{equation}
\label{max_reg_estim}
\Vert \dot{u}\Vert_{L^p([0,T],\mathscr{X})}+\Vert u\Vert_{L^p([0,T],\mathscr{X})}+\Vert \mathscr{A} u\Vert_{L^p([0,T],\mathscr{X})}\leq \kappa \Vert f\Vert_{L^p([0,T],\mathscr{X})}
\end{equation}
for a constant $\kappa:=\kappa(p)>0$ independent of $f$.\\
It is known that a necessary condition for the maximal $L^p$-regularity is that $\mathscr{A}$ generates an analytic semigroup $\mathscr{T}:=(\mathscr{T}(t))_{t\ge 0}$. This condition is also sufficient if $\mathscr{X}$ is a Hilbert space, see De Simon \cite{Simon}. Moreover,  it is shown in \cite{Dore} that if  $\mathscr{A}$ has maximal $L^p$-regularity for one $p\in [1,\infty]$ then $\mathscr{A}$ has maximal $L^q$-regularity for all $q\in]1,\infty[$.

Next we recall a perturbation result on maximal $L^p$-regularity. To this end, we need the following concept.
\begin{definition}\label{obs-admi}
  let $\mathscr{A}$ be the generator of a strongly continuous semigroup $(\mathscr{T}(t))_{t\ge 0}$ on a Banach space $\mathscr{X}$, and let $\mathscr{Y}$ be another Banach space. An operator $\mathscr{C}\in \mathcal{L}(D(\mathscr{A}),\mathscr{Y})$ is called $p$-admissible observation operator for $\mathscr{A}$, if there exist (hence all) $\alpha >0$ and a constant $\gamma:=\ga(\al)>0$ such that:
		\begin{equation}\label{walid}
		\int_0^\alpha \Vert \mathscr{C}\mathscr{T}(t)x\Vert_{\mathscr{Y}}^p dt \leq \ga^p \Vert x \Vert^p,
		\end{equation}
		for all $x\in D(\mathscr{A})$. We also say that $(\mathscr{C},\mathscr{A})$ is $p$-admissible.
\end{definition}

The following theorem gives an invariance result of the maximal $L^p$-regularity, see \cite{AmBoDrHa1}.
\begin{theorem}\label{theorem:autonomous Miyadera}
	If $\mathscr{P}\in \calL(D(\mathscr{A}),\mathscr{X})$ is $l$-admissible for $\mathscr{A}$ for some $l>1$ and $\mathscr{A}$ has maximal $L^p$-regularity
	then the operator $\mathscr{A}+\mathscr{P}:D(\mathscr{A})\to \mathscr{X}$ is so.
\end{theorem}

\section{Maximal regularity for Volterra integro-differential equations}\label{sec:autonomous volterra witout BC}
In this section, we study the maximal $L^p$-regularity of the autonomous Volterra integro-differential equation \eqref{IACP}. First, certain conventions are defined. Let the Banach product space
\begin{align*}
\mathscr{X}:=X\times L^q(\R^+,X)\quad\text{with norm}\quad \left\|(\begin{smallmatrix} x\\ f\end{smallmatrix})\right\|:=\|x\|+\|f\|_q.
\end{align*}
Let us define the left shift semigroup on $L^q(\R^+,X)$ by
\begin{align*}
(\Ss(t)f)(s)=f(t+s),\qquad t,s\ge 0.
\end{align*}
We select
\begin{align}\label{matrix-opetra1}\begin{split}
& \Upsilon x := a(\cdot)Fx,\qquad x\in D(\A),\\
&\mathfrak{A} :=\left(\begin{array}{cc}
\A & \delta_0\\
\Upsilon & \frac{d}{ds}
\end{array}\right),\qquad
D(\mathfrak{A})=D(\A)\times D(\frac{d}{ds}).
\end{split}
\end{align}
Moreover, we consider the function
\begin{align}\label{zeta-function}
\zeta: [0,\infty)\to \mathscr{X},\quad\zeta(t)=\left(\begin{smallmatrix} f(t)\\ 0\end{smallmatrix}\right),\qquad t\in [0,T].
\end{align}

Let $z:[0,\infty)\to X$ satisfies \eqref{IACP}. According to \cite[Section VI.7]{EngNag}, the Volterra equation \eqref{IACP} can be reformulated as the following non-homogeneous Cauchy problem on $\mathscr{X},$
\begin{align}
\label{ACP1}
\begin{cases}
\dot{\varrho}(t) = \mathfrak{A} \varrho(t) + \zeta(t), & t\in [0,T], \\
\varrho(0) = (\begin{smallmatrix} 0\\ 0\end{smallmatrix}),
\
\end{cases}
\end{align}
where
\begin{align*}
\varrho(t)=\big(\begin{smallmatrix} z(t)\\ g(t, \cdot)\end{smallmatrix}\big),\qquad g(t,\cdot)=\mathbb{S}(t)g(0,\cdot)+\int^t_0 \mathbb{S}(t-s)\Upsilon z(s) ds,\qquad t\ge 0.
\end{align*}
In order to study the maximal $L^p$-regularity of \eqref{IACP}, it suffices to study the one of the linear operator $\mathfrak{A}$. To this end, we will use Theorem \ref{theorem:autonomous Miyadera}. In fact, we first split the operator $\mathfrak{A}$ as
\begin{align}\label{split-A}
\mathfrak{A}=\mathscr{A}+\mathscr{P}
\end{align}
where
\begin{align}\label{calAcalP}
\begin{split}
&\mathscr{A} :=\begin{pmatrix}
\A & 0\\
0 & \frac{d}{ds}
\end{pmatrix},\quad
D(\mathscr{A})=D(\A)\times W^{1,q}(\R^+,X),\cr
&\mathscr{P} :=\begin{pmatrix}
0 & \delta_0\\
\Upsilon & 0
\end{pmatrix},\quad
D(\mathscr{P})=D(\mathscr{A}).
\end{split}
\end{align}
Clearly, the operator $\mathscr{A}$ generates the following $C_0$-semigroup on $\mathscr{X},$
\begin{align*}
\mathscr{T}(t)=\begin{pmatrix}\T(t)&0\cr 0& \Ss(t)\end{pmatrix},\qquad t\ge 0.
\end{align*}
Moreover, $\mathscr{T}$ is not an analytic semigroup on $\mathscr{X}$ even if we assume that the semigroup $\T$ is analytic on $X$. This is due to the fact that the left shift semigroup $\mathbb{S}:=(\Ss(t))_{t\ge 0}$ is not analytic in $L^q(\R^+,X)$. This means that $\mathscr{A}$ has not the maximal $L^p$-regularity on the space $\mathscr{X}$. To overcome this problem, one way is to look for subspaces of $L^q(\R^+,X)$ in which the shift semigroup $\mathbb{S}$ is analytic. As shown  \cite{Barta1} and \cite{Barta2} a perfect space in which the shift semigroup is analytic is the Begrman space which we define as follows:

\begin{definition}\label{Def-Bergman-space}
  For $q\in(1,\infty)$, we define the {\em Bergman space} of holomorphic $L^q$-integrable functions by:
\begin{equation*}
B^q(\Sigma_\theta;X):=\left\{f:\Sigma_\theta\rightarrow X\ \text{holomorphic}\ ;\ \int_{\Sigma_\theta}\Vert f(\tau+i\si)\Vert_{X}^q d\tau d\si<\infty\right\}.
\end{equation*}
where 
\begin{equation*}
	\Sigma_\theta:=\left\{\la\in \mathbb{C}:|\arg(\la)|<\theta \right\},\qquad 0<\theta\le \frac{\pi}{2}.
\end{equation*}
This space is also defined by $B_{\theta,X}^q$. Moreover, if $X=\mathbb{C}$, then we write $B_{\theta}^q$ instead of  $B_{\theta,\mathbb{C}}^q$.
\end{definition}
The space $B_{\theta,X}^q$, endowed with the following norm
\begin{equation*}
\|f\|_{B_{\theta,X}^q}:=\left(\int_{\Sigma_\theta}\Vert f(\tau+i\si)\Vert_{X}^q d\tau d\si\right)^\frac{1}{q},
\end{equation*}
is a Banach space.

The proof of the following result can be found in \cite{Barta2} and \cite{Barta3}.
\begin{proposition}\label{analytic-shift}
  The complex derivative $\displaystyle{\frac{d}{dz}}$ with its natural domain :
$$
D\left(\frac{d}{dz}\right):=\left\{ f\in B_{\theta,X}^q; f' \in B_{\theta,X}^q\right\}
$$
generates an analytic semigroup of translation on $B_{\theta,X}^q$. Furthermore, if $X$ is an UMD space, then $\displaystyle{\frac{d}{dz}}$ enjoys the maximal $L^p$-regularity on $B_{\frac{\pi}{2},X}^q$.
\end{proposition}

This result motivated us to replace the state space $\mathscr{X}$ with the following appropriate space
\begin{align*}
\mathscr{X}^q:=X\times B_{\theta,X}^q,\qquad\left\|\left(\begin{smallmatrix} x\\f\end{smallmatrix}\right)\right\|_{\mathscr{X}^q}:=\|x\|_{X}+ \|f\|_{B_{\theta,X}^q}.
\end{align*}
According to Proposition \ref{analytic-shift}, the following result becomes trivial.
\begin{lemma}\label{L1}
Assume that $X$ is an UMD space. If $\A$ has the maximal $L^p$-regularity in $X$, then the operator $\mathscr{A}$ defined in \eqref{calAcalP} has the maximal $L^p$-regularity in $\mathscr{X}^q$.
\end{lemma}

Now if $\displaystyle\frac{d}{dz}$ has the maximal $L^p$-regularity on $B_{\theta,X}^q$ for some $\theta\in (0,\displaystyle\frac{\pi}{2}]$ then according to Theorem \ref{theorem:autonomous Miyadera} and the decomposition \eqref{split-A}, $\mathfrak{A}$ will have the  maximal $L^p$-regularity on $\mathscr{X}^q$ as long as one proves that the perturbation $\mathscr{P}$ is $l$-admissible for $\mathscr{A}$ for some $l\in (1,\infty)$.

We have the following technical result.
\begin{lemma}\label{L2}
Assume that $a(\cdot)\in B^q_\theta$ for $0<\theta\leq\frac{\pi}{2}$ and let $F\in\calL(D(\A),X)$ be $p$-admissible for $\A$. Then
\begin{align}\label{last-estimate}
\int^\alpha_0 \left\|\mathscr{P}\mathscr{T}(t)(\begin{smallmatrix} x\\f\end{smallmatrix})\right\|^pdt\le
2^{p-1}\left( \|a\|_{B_{\theta}^q}^p\ga^p \|x\|^p_X+\int^\alpha_0\|f(t)\|^pdt\right)
\end{align}
$\al>0$ and $(\begin{smallmatrix} x\\f
\end{smallmatrix})\in D(\mathscr{A})$.
\end{lemma}
\begin{proof}
Let  $\al>0$ and $(\begin{smallmatrix} x\\f
\end{smallmatrix})\in D(\mathscr{A}),$ we have
\begin{align*}
\int^\alpha_0 \left\|\mathscr{P}\mathscr{T}(t)(\begin{smallmatrix} x\\f\end{smallmatrix})\right\|^pdt&\le 2^{p-1}\left(\int^\alpha_0 \|\Upsilon\T(t)x\|^p_{B_{\theta}^q}dt+\int^\alpha_0\|f(t)\|^pdt\right)\cr & \le  2^{p-1}\left( \|a\|_{B_{\theta}^q}^p\int^\al_0  \|F\T(t)x\|_{X}^p dt+\int^\alpha_0\|f(t)\|^p dt\right).
\end{align*}
Now the estimate \eqref{last-estimate} immediately follows the $p$-admissible of $F$ for $\A.$
\end{proof}

According to Lemma \ref{L2}, To prove the $p$-admissibility of $\mathscr{P}$ for $\mathscr{A}$ for some $p\in (1,\infty)$ it suffices to estimate the $L^p$ norm of $f$ by its norm on the Bergman space $B_{\theta,X}^q$. To this end, we need the following lemma inspirited from \cite[lem.4.3]{Barta3} (here we slightly modify the result proved in \cite[lem.4.3]{Barta3} and give a sharp estimate.)
\begin{lemma}\label{lemma:bergman_estimate}
	Let $s\in (1,2)$ and  $q>2$. For $p_{s,q}: = \displaystyle{\frac{q(s-1)}{s}}\;(>1)$ and $f\in B^q_{\thh,X}$ for some $\theta\in (0,\displaystyle\frac{\pi}{2}]$, we have
	\begin{equation*}
	\left(\int_0^R \|f(t)\|_X^{p_{s,q}} dt \right)^\frac{1}{{p_{s,q}}}\leq C_R \|f\|_{B^q_{\thh,X}},
	\end{equation*}
	for all $R>0$ and $C_R>0$ only depends on $R$ that verifies $C_R\to 0$ as $R\to 0$.
\end{lemma}
\begin{proof} First, let $\theta\in (0,\displaystyle\frac{\pi}{2})$
	and let us estimate the value of $\|f(t)\|^{p_{s,q}}$ using the Cauchy formula. The integration path will consist of two circle segments (see the next Figure). Let $\ga_1(t) := r-acr +are^{it}$ and $\ga_2(t) := r+acr -are^{it}$, $t\in [-\alpha, \alpha]$ with  $c :=\cos \alpha$ and $a=\tan(\theta)$ such that $ac<1$.
	
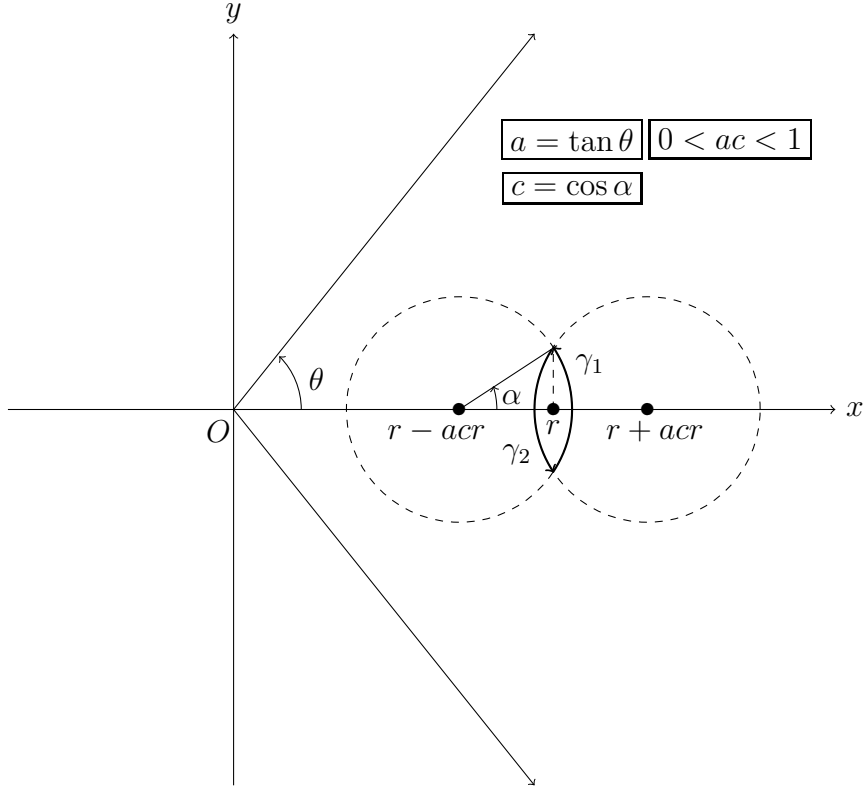
\begin{center}
		\begin{tikzpicture}
			\draw[->] (-3,0) -- (8,0);
			\draw (8,0) node[right] {$x$};
			\draw [->] (0,-5) -- (0,5);
			\draw (0,5) node[above] {$y$};
			\draw (-0.2,0) node[below] {$O$};
			
			\draw [dashed](3,0) circle (1.5);
			\draw (3,0) node {$\bullet$} ;
			\draw (4.25,0) node {$\bullet$} ;
			\draw (4.25,0) node[below] {$r$};
			\draw (2.7,0) node[below] {$r-acr$};
			\draw [dashed](5.5,0) circle (1.5);
			\draw (5.5,0) node {$\bullet$} ;
			\draw (5.6,0) node[below] {$r+acr$};
			\draw (3,0) --(4.3,0.85);
			\draw [->](0,0) --(4,5);
			\draw [dashed](4.25,0) --(4.25,0.85);
			\draw [->](0,0) --(4,-5);
			\draw [->](0.9,0) arc (0:45:1);
			\draw [->](3.5,0) arc (0:18:1);
			\draw (1.1,0.7) node[below] {$\theta$};
			\draw (3.7,0.4) node[below] {$\alpha$};
			\draw [->][black,thick,domain=-34:34] plot ({3+1.5*cos(\x)}, {1.5*sin(\x)});
			\draw [->][black,thick,domain=148:214] plot ({5.5+1.5*cos(\x)}, {1.5*sin(\x)});
			
			\draw (4.4,0.6) node[right] {$\gamma_1$};
			\draw (4.1,-0.6) node[left] {$\gamma_2$};
			\draw (4.5,4) node[below] {\fbox{$a=\tan \theta$}};
			\draw (4.5,3.3) node[below] {\fbox{$c=\cos \alpha$}};
			\draw (6.6,4) node[below] {\fbox{$0<ac<1$}};
		\end{tikzpicture}
	\end{center}

	\begin{align*}
	f(r) &= \frac{1}{2i\pi} \int_\ga \frac{f(z)}{z-r}dz\cr
	&= \frac{1}{2i\pi} \left( \int_{-\alpha}^\alpha \frac{f(r-arc + are^{it})iare^{it}}{-arc + are^{it}} dt + \int_{-\alpha}^\alpha \frac{f(r+arc - are^{it})iare^{it}}{arc - are^{it}} dt \right)
	\end{align*}
	then
	\begin{align*}
	&\|f(r)\|^{p^s_{q}} \leq \frac{1}{(2\pi)^{p^s_{q}}} \left( \int_{-\alpha}^\alpha \frac{\|f(r-arc + are^{it})\|}{ 1-c } dt + \int_{-\alpha}^\alpha \frac{\|f(r+arc-are^{it})\|}{1-c} dt \right)^{p^s_{q}}\cr
	&\leq\frac{2^{{p_{s,q}}-1}}{\left(2\pi(1-c)\right)^{p_{s,q}}}\left(\left(\int_{-\alpha}^\alpha \|f(r-arc + are^{it})\| dt\right)^{p_{s,q}}+\left(\int_{-\alpha}^\alpha \|f(r+arc - are^{it})\| dt\right)^{p^s_{q}}   \right)\cr
	&\leq  \frac{(4\alpha)^{{p_{s,q}}-1}}{(2\pi(1-c))^{p_{s,q}}}\left(\int_{-\alpha}^\alpha \|f(r-arc + are^{it})\|^{p_{s,q}} dt+\int_{-\alpha}^\alpha \|f(r+arc - are^{it})\|^{p_{s,q}} dt\right)
	\end{align*}
	Now we shall estimate the first integral in the above inequality, and the second one can be estimated in a similar way. We set $\psi(t,r)=(r-arc+ar\cos t, ar \sin t)$ then the Jacobian of $\psi$ satisfies
	\begin{align*}
	|J_\psi| &=\left|\begin{matrix}
	1-ac+a\cos t & -ar\sin t \\
	a\sin t      & ar\cos t
	\end{matrix} \right| = ar((1- ac) \cos t + a)\cr
	&\geq ar(c(1-ac) + a) \cr
	\end{align*}
	Since $x= r+ar(\cos t -c) < r(1+a(1-c))$, we have
	\begin{align*}
	|J_\psi| \geq ax \frac{c(1-ac) +a}{1+a(1-c)} := c_1 x
	\end{align*}
	then
	$$
	\int_0^R\int_{-\alpha}^\alpha \|f(\psi(t,r)) \|^{p_{s,q}} dt dr = \int \int_M \frac{\|f(x+iy) \|^{p_{s,q}}}{|J_\psi|}dx dy
	$$
	where $M:=\psi((0,R)\times [-\alpha,\alpha])$ is contained in $M':=\{ x+iy\in \C; 0<x<R+\delta \text{ whith } \delta := R(1-c)a \text{, and } |y|\leq ar \sin \alpha < ax \sin \alpha \}$ . This inclusion and H$\ddot{o}$lder inequality imply that
	\begin{align*}
	\int& \int_M \frac{\|f(x+iy) \|^{p_{s,q}}}{|J_\psi|}dx dy \leq \frac{1}{c_1} \int \int_{M'} \frac{\|f(x+iy) \|^{p_{s,q}}}{x}dx dy\cr
	& \leq \frac{1}{c_1} \left(\int \int_{M'} \|f(x+iy) \|^{{p_{s,q}}s'}\right)^\frac{1}{s'} \left(\int \int_{M'} x^{-s}\right)^\frac{1}{s} \cr
	& \leq \frac{1}{c_1} \left(\int \int_{M'} \|f(x+iy) \|^{{p_{s,q}}s'}\right)^\frac{1}{s'} \left(\int_0^{R+\delta} x^{-s} \int_0^{ax\sin \alpha}2 dy dx \right)^\frac{1}{s} \cr
	& \leq \frac{(2a\sin \alpha)^\frac{1}{s}}{c_1} \|f\|_{B^{{p_{s,q}}s'}_{\thh,X}}^{p_{s,q}}  \left(\int_0^{R+\delta} x^{1-s} dx \right)^\frac{1}{s}
	\leq \frac{(2a\sin \alpha)^\frac{1}{s}}{c_1} \|f\|_{B^{{p_{s,q}}s'}_{\thh,X}}^{p^s_{q}}  \frac{(R+\delta)^{\frac{2}{s}-1}}{(2-s)^\frac{1}{s}}\cr
	& = (2a\sin \alpha)^\frac{1}{s}  \frac{1+a(1-c)}{a(a+c(1-ac))} \frac{(R(1+(1-c)a))^{\frac{2-s}{s}}}{(2-s)^\frac{1}{s}}\|f\|_{B^{{p^s_{q}}s'}_{\thh,X}}^{p_{s,q}} \cr
	& :=\tilde{C} R^{\frac{2-s}{s}} \|f\|_{B^{{p_{s,q}}s'}_{\thh,X}}^{p_{s,q}}.
	\end{align*}
	where the constant $\tilde{C}$ does not depend on $R$.\\
	Finally we have
	$$
	\left(\int_0^R \|f(t)\|^{p_{s,q}} dt \right)^\frac{1}{{p_{s,q}}} \leq C_{R} \|f\|_{B^{q}_\thh}
	$$
	with $C_R\to 0$ as $R\to 0$.\\
	The case of $\theta=\frac{\pi}{2}$ follows easily due to the fact that the space $B_{\theta,X}^q$ is decreasing  with respect to $\theta$ with continuous injection.
\end{proof}
\begin{remark}\label{pqs}
	Given $q,l>1$, there always exists $s_{q,l}\in (1,2)$ such that: $$1<p_{s_{q,l}}:=\displaystyle{\frac{q(s_{q,l}-1)}{s_{q,l}}}\le l.$$ In fact, if $q\in (1,l]$ the assertion is trivial. Now if $q\in [2l, \infty)$ we have $\displaystyle{\frac{q}{q-l}}\in (1,2)$. Hence all $s_{q,l}\in  (1,\displaystyle{\frac{q}{q-l}}]$ satisfy the required estimation. Finally, for $q\in (l,2l]$ we have $2\le\displaystyle{\frac{q}{q-l}}$. Thus all $s\in (1,2)$ will satisfy the estimation. The fact that the space of $l$-admissible operators is decreasing with respect to the exponent $l$, the discussion above shows that a sufficient condition to have the required $p_{s_{q,l}}$-admissibility for $A$ in Theorem \ref{good-thm} is in fact the $l$-admissibility for some $l> 1.$
\end{remark}
Now we state the first result of this paper:
\begin{theorem}\label{good-thm}
	Let $X$ be a UMD space and that $a(\cdot)\in B_{\theta}^q$ for some $q>2$ and $\theta\in (0,\pi/2]$ and $F\in\calL(D(\A),X)$ is a ${l_0}$-admissible observation operator for $\A$ for some $l_0\in (1, \infty)$. If both $\A$ and $\displaystyle\frac{d}{dz}$ have the maximal $L^p$-regularity in $X$ and $B_{\theta,X}^q$  respectively, then $\mathfrak{A}$ has the maximal $L^p$-regularity on $\mathscr{X}^q$. Moreover, if $p\in (1,l_0]$ and $z$ is the solution of the problem \eqref{IACP}, then there exists $C_p>0$ independent of  $f\in L^p([0,T],X)$ such that
	\begin{align}\label{ESTIMR}
	\|\dot{z}\|_{L^p([0,T],X} + \|\A z\|_{L^p([0,T],X} +\|z\|_{L^p([0,T],X)} \leq C_p \|f\|_{L^p([0,T],X)}.
	\end{align}
\end{theorem}
\begin{proof}
	The proof uses Theorem \ref{theorem:autonomous Miyadera} and the decomposition $\mathfrak{A}=\mathscr{A}+\mathscr{P}$ given in \eqref{split-A}. Lemma \ref{L1} shows that the operator $\mathscr{A}$ has the maximal $L^p$-regularity on $\mathscr{X}^q$. Let $s_{q,l_0}$ and $p_{s_{q,l_0}}$ as in Remark \ref{pqs}. Now by combining Lemma \ref{L2} and Lemma \ref{lemma:bergman_estimate}, it is clear that the operator $\mathscr{P}$ is $p_{s_{q,l_0}}$-admissible for $\mathscr{A}$. Appealing to Theorem \ref{theorem:autonomous Miyadera}, the operator $\mathfrak{A}$ also enjoys the maximal $L^{p_{s_{q,l_0}}}$-regularity on $\mathscr{X}^q$. It is well known that if an operator has maximal $L^p$-regularity for some $p\in (1, \infty)$, then it has maximal $L^p$-regularity for all $p\in (1, \infty)$ (see for instance \cite{Dore}) and hence $\mathfrak{A}$ has the maximal $L^p$-regularity on $\mathscr{X}^q$. Thus there is a constant $C_p>0$ such that
	$$
	\|\dot{\varrho}\|_{L^p([0,T],\mathscr{X}^q)} + \|\mathfrak{A}\varrho\|_{L^p([0,T],\mathscr{X}^q)} +\|\varrho\|_{L^p([0,T],\mathscr{X}^q)}\leq C_p\|f\|_{L^p([0,T],X)}.
	$$
	Since
	$
	\mathfrak{A}\varrho(t) = \left(\begin{smallmatrix} \A z(t) + g(t,0)\\ \Upsilon z(t) + \displaystyle{\frac{dg(t,\cdot)}{dz}} \end{smallmatrix}\right)
	\quad\text{
		and}
	\quad
	\dot{\varrho}(t) = \left(\begin{smallmatrix} \dot{z}(t)\\ \dot{g}(t,\cdot) \end{smallmatrix}\right),
	$
	we have
	\begin{align*}
	&\|\dot{z}\|_{L^p([0,T],X)} + \|\A z\|_{L^p([0,T],X)} +\|z\|_{L^p([0,T],X)}\cr
	&\leq \|\dot{z}\|_{L^p([0,T],X)} + \| \A z + g(\cdot,0)\|_{L^p([0,T],X)} + \| g(\cdot, 0)\|_{L^p([0,T],X)} +\|z\|_{L^p([0,T],X)}\cr
	&\leq C\|f\|_{L^p([0,T],X)} + \| g(\cdot, 0)\|_{L^p([0,T],X)}.
	\end{align*}
	On the other hand, we have
	$
	\dot{g}(t,\cdot) = \Upsilon z(t) + \displaystyle{\frac{dg(t,\cdot)}{dz}},
	$
	hence
	$$
	g(t,\cdot) = \Ss(t) g(0,\cdot) + \int_0^t \Ss(t-s) a(\cdot)Fz(s) ds.
	$$
	Thus
	\begin{align*}
	\| g(\cdot, 0)\|_{L^p([0,T],X)}^p &= \int_0^T \| g(t, 0)\|_X^p dt\cr
	&= \int_0^T\|\int_0^t a(t-s)Fz(s)\|^p\cr
	&\leq T^{p-1} \int_0^T \int_0^T |a(t-s)|^p \|Fz(s)\|^p ds dt\cr
	&= T^{p-1} \int_0^T \int_s^T |a(t-s)|^p \|Fz(s)\|^p dt ds\cr
	&= T^{p-1} \int_0^T \left(\int_s^T |a(t-s)|^pdt\right) \|Fz(s)\|^p ds \cr
	&= T^{p-1} \int_0^T \left(\int_0^{T-s} |a(t)|^pdt\right) \|Fz(s)\|^p ds \cr
	&\leq T^{p-1}  \left(\int_0^{T} |a(t)|^pdt\right) \int_0^T \|Fz(s)\|^p ds \cr
	&\leq T^{p-1} C_T^p \|a\|_{B^q_{\theta}}^p   \int_0^T \|Fz(s)\|^p ds.
	\end{align*}
	where $C_T>0$ is the constant in Lemma \ref{lemma:bergman_estimate}.\\ Note that $z(t) \in D(\A)$ for almost every $t\in [0, T]$, due to the maximal $L^p$-regularity of $\mathfrak{A}$. This shows that
	\begin{align*}
	\| g(\cdot, 0)\|_{L^p([0,T],X} \leq T^{\frac{p-1}{p}} C_T \|a\|_{B^q_{\theta}}  \|Fz\|_{L^p([0,T],X)}.
	\end{align*}
	It suffices only to estimate $\|Fz\|_{L^p([0,T],X)}$ by $\|f\|_{L^p([0,T],X)}$. In fact, we know that
	$$
	\dot{z}(t)= \A z(t) +g(t,0)+f(t),\quad t\in [0,T],
	$$
	which gives
	$$
	Fz(t)= F\int_0^t \T(t-s)g(s,0)ds + F\int_0^t \T(t-s)f(s)ds,\quad \text{ a.e. }t\in [0,T].
	$$
	Keeping in mind that the space of $l$-admissible operators is decreasing with respect to the exponent $l$ and that $p\le l_0$ implies that
	$F$ is $p$-admissible for $\mathbb{A}$. Therefore, by using \cite[Prop 3.3]{Hadd},
	\begin{align*}
	\|Fz\|_{L^p([0,T],X)}&\leq  \ga_T \|g(\cdot,0)\|_{L^p([0,T],X)} + \ga_T \|f\|_{L^p([0,T],X)}\cr
	&\leq \ga_T T^{\frac{p-1}{p}} C_T \|a\|_{B^q_{\theta}}  \|Fz\|_{L^p([0,T],X)} + \ga_T \|f\|_{L^p([0,T],X)},
	\end{align*}
	where $\ga_T>0$ is the constant satisfying $\ga_T\to 0$ as $T\to 0$. Hence, if we choose $T$ such that $\beta_T :=\ga_T T^{\frac{p-1}{p}} C_T \|a\|_{B^q_{\theta}}<1$, we obtain
	$$
	\|Fz\|_{L^p([0,T],X)}\leq \frac{\ga_T}{1-\beta_T}  \|f\|_{L^p([0,T],X)}.
	$$
	Hence we have
	$$
	\|\dot{z}\|_{L^p([0,T],X)} + \| \A z\|_{L^p([0,T],X)} +\|z\|_{L^p([0,T],X)} \leq C \|f\|_{L^p([0,T],X)}.
	$$
\end{proof}
\begin{remark}  In contrast to Barta's result and as for Cauchy problem, we have obtained the estimation \eqref{ESTIMR}. One wonders if the method used in \cite{Barta2} can prove the aforementioned estimate.
\end{remark}
\begin{example}
	In this section we investigate the maximal $L^p$-regularity of the following Volterra integro-differential heat equation involving a fractional power of the Laplace operator
	
	\begin{equation}\label{Dirichlet}
	\begin{cases}
	\dot{u}(t,x) = (\Delta + P) u(t,x)+\displaystyle\int_0^t  \beta e^{-\gamma(t-s)}(-\Delta)^\al u(s,x)ds +f(t) , & \mbox{ } t\in [0,T], x\in \Omega \\
	u(t,x) = 0, & \mbox{} t\in [0,T], x\in\partial\Omega\\
	u(0,x) = 0, & \mbox{} x\in \Omega
	\
	\end{cases}
	\end{equation}
	where $\Omega\in \mathbb{R}^n$ is a bounded Lipschitz domain and $\al \in (0,1/2]$ and $\beta, \gamma >0$. It is an example of anomalous equation of diffusion type.
	Let us first verify that the following Dirichlet-Laplacian operator defined on suitable $L^r(\Omega)$ by
	$$D(\Delta^D_r)= W^{2.r}(\Om) \cap W_0^{1,r}(\Om),$$
	$$\Delta^D_r=\Delta,$$
	for certain domains $\Omega$ and a range of exponents $r$, has the maximal $L^p$-regularity. For $n\ge 2$ and $1<r\le 2$, it is shown in \cite{Wood07} that for an $\Omega$ satisfying a uniform outer ball condition, the operator $\Delta^D_r$ generates a positive, contractive and exponentially stable $C_0$-semigroup on $L^r(\Omega)$ enjoying the maximal $L^2$-regularity. Therefore, due to Kalton and Weis' result \cite[ Corollary 5.2]{KW01},  the Dirichlet operator $-\Delta^D_r$ admits a bounded $H^{\infty}(\Sigma_\theta )$ functional calculus with $\theta <\pi/2$. The fact that $L^r(\Omega)$ is of cotype $2$, thanks to \cite[Theorem 4.2]{LeMerdy}, the fractional power $(-\Delta^D_r)^{1/2}$ is $2$-admissible for $\Delta^D_r$. Now assume that the unbounded operator $P$ satisfies the following resolvent estimate
	$P:D(\Delta^D_r)\to L^r(\Omega)$ such that
	$$
	\|\sqrt{\la}PR(\la,\Delta^D_r) \|\leq M
	$$
	for some constant  $M>0$ and for all $\la >0$.\\
	In view of \cite[Theorem 4.1]{LeMerdy}, the operator $P$ is $2$-admissible for $\Delta^D_r$. Now Theorem 2.2 show that $\Delta^D_r+P$ enjoys the maximal $L^2$-regularity on  $L^r(\Omega)$. Since both $P$ and $(-\Delta^D_r)^{1/2}$ are $2$-admissible for $\Delta^D_r$, we deduce that $(-\Delta^D_r)^{1/2}$ is $2$-admissible for $\Delta^D_r+P$ (see. \cite{Hadd}) and in virtue of Theorem \ref{good-thm} and Remark \ref{pqs} we conclude that the problem \eqref{Dirichlet} has maximal $L^2$-regularity $L^r(\Omega)$ and the estimation \eqref{ESTIMR} takes place in particular for $p=2$.\\
	For a bounded Lipschitz (or convex) domain and $n\ge 3$, similar result can now be also obtained for the Neumann-Laplacian defined on $L^r(\Omega)$ by
	$$D(\Delta^N_r)= \left\{ u\in  W^{2.r}(\Om):\frac{\partial u}{\partial \nu}=0 \; on\; \partial \Omega\right\}. $$
	$$\Delta^N_r=\Delta.$$
	Indeed, in virtue of \cite[Theorem 6.4]{Wood07} and by proceeding in a very similar way as for Dirichlet boundary conditions we obtain the maximal $L^p$-regularity result for the above integro-differential equation with Neumann-Laplacian. For $\al \in(0,\displaystyle\frac{1}{2})$ the result follows in a similar way since analyticity shows that $(-\Delta^D_r)^{\al}$ is always $2$-admissible for $\Delta^D_r$.
\end{example}

\section{Maximal regularity for boundary Volterra integro-differential equations}\label{sec:volterra with BC}
Let $X,U$ and $Z$ be Banach spaces such that $Z\subset X$ continuously and densely. Let $A_m:D(A_m):=Z\to X$ be a closed linear operator, and let $F:Z\to X,$  $G,K:Z\to U$ be linear operator.

The object of this section is to investigate the maximal $L^p$-regularity of the problem \eqref{InDiffwBC}
\begin{align*}
\begin{cases}
\dot{z}(t)=A_m z(t)+\displaystyle\int_0^t a(t-s) F z(s) ds + f(t),& t\in [0,T],\cr
z(0)=0,\cr Gz(t)=Kz(t),& t\in [0,T].
\end{cases}
\end{align*}
We introduce the linear operator
\begin{align}\label{closed}
\A:=A_m,\qquad D(\A)=\left\{x\in Z:Gx=Kz\right\}.
\end{align}
and we set $\calP:= Fi$ where $i$ the continuous injection from $D(\A)$  to $Z$. Then the equation \eqref{InDiffwBC} is similar to the equation \eqref{IACP} with $\Upsilon := a(\cdot)\calP$. Thus, to prove maximal $L^p$-regularity of \eqref{InDiffwBC} it suffices to find conditions for which $\A$ has maximal regularity and $\Upsilon$ is admissible for $\A$.

Throughout this section we assume that the operator $G$ is surjective and
\begin{align*}
A:=(A_m)_{|D(A)} \quad\text{with}\quad D(A):=\ker(G),
\end{align*}
generates a strongly continuous semigroup $T:=(T(t))_{t\ge 0}$ on $X$. We denote by $\rho(A)$ the resolvent set of $A$, $R(\la,A)=(\la-A)^{-1},\;\la\in\rho(A),$ the resolvent operator of $A$. We also consider a new norm on $X$ defines by $\|x\|_{-1}:=\|R(\mu,A)x\|$ for $x\in X$ and $\mu\in\rho(A)$ (this norm is independent of the choice of $\mu,$ due to the resolvent equation). We denote by $X_{-1}$ the completion of $X$ with respect to the norm $\|\cdot\|_{-1},$ which a Banach space, called the extrapolation space associated with $X$ and $A$. We have the following continuous and dense embedding $D(A)\subset X\subset X_{-1}$. The semigroup $T$ is extended to another strongly continuous semigroup $T_{-1}:=(T_{-1}(t))_{t\ge 0}$ on $X_{-1},$ whose generator $A_{-1}:X\to X_{-1}$ is the extension of $A$ to $X$, see \cite[Chapter II]{EngNag}.

According to Greiner \cite{Grei}, for each $\la\in\rho(A),$ the restriction of $G$ to $\ker(\la-A_m)$ is invertible with inverse $\D_\la$ (called the Dirichlet operator) given by
\begin{align*}
\D_\la:=\left(G_{|\ker(\la-A_m)}\right)^{-1}\in\calL(U,X).
\end{align*}
Now, define the operators
\begin{align*}
&B:=(\la-A_{-1})\D_\la\in\calL(U,X_{-1}),\quad\la\in\rho(A),\cr & C:=Kj\in\calL(D(A),U),\cr
&\Pp:=Fj\in\calL(D(A),X)
\end{align*}
where $j$ is the continuous injection from $D(A)$ to $Z$. Due to the resolvent equation, the operator $B$ is independent of $\la$.

We also need the operators
\begin{align*}
\Phi_t u:=\int^t_0 T_{-1}(t-s)Bu(s)ds,\qquad t\ge 0,\quad u\in L^p([0,+\infty),U).
\end{align*}
This integral takes its values in $X_{-1}$. However, by using an integration by parts, for any $t\ge 0$ and $u\in W^{2,p}_{0,t}(U),$ where
\begin{align*}
W^{2,p}_{0,t}(U):=\left\{u\in W^{2,p}([0,t],U):u(0)=\dot{u}(0)=0\right\},
\end{align*}
we have
\begin{align*}
\Phi_t u=\D_0 u(t)-\int^t_0 T(t-s)\D_0 \dot{u}(s)ds\in Z,
\end{align*}
where we assumed that $0\in\rho(A)$ (without loss of generality).

Thus the following operator is well defined
\begin{align*}
(\F u)(t)=K \Phi_t u,\qquad u\in W^{2,p}_{0,t}(U),\quad t\ge 0.
\end{align*}
Here, we make the following assumption
\begin{itemize}
  \item [{\bf (H)}] The triple $(A,B,C)$ is $p$-regular, $p\in (1,\infty)$, with $I_U:U\to U$ as an admissible feedback. That is the following assertions hold:
  \begin{enumerate}
    \item $C$ is a $p$-admissible observation operator for $A$ (see Definition \ref{obs-admi}),
    \item $B$ is a $p$-admissible control operator for $A$. This means that there exists $\tau>0$ such that $\Phi_{\tau} u\in X$ for any $u\in L^p([0,+\infty),U),$
    \item For any $\tau>0,$ there exists $\kappa:=\kappa(\tau)>0$ such that
    \begin{align*}
    \|\F u\|_{L^p([0,\tau],U)}\le \kappa \|u\|_{L^p([0,\tau],U)},\qquad \forall u\in W^{2,p}_{0,\tau}(U).
    \end{align*}
    (hence we can extend $\F$ to a bounded operator on $L^p([0,\tau],U)$ for any $\tau>0$).
    \item The following limit exists in $U$ for any $v\in U$:
    \begin{align*}
    \lim_{h\to 0^+}\frac{1}{h}\int^h_0 \left(\F (\1_{\R^+} v)\right)(s)ds=0.
    \end{align*}
    \item $1\in\rho(\F)$.
  \end{enumerate}
\end{itemize}
We have the following observations about the condition {\bf(H)}.
\begin{remark}\label{remark-about-H}
(i) If $C$ is bounded from $X$ to $U$ (hence $C=K\in\calL(X,U)$), and $B$ is $p$-admissible for $A,$ then the condition {\bf(H)} is satisfied.\\
(ii)  Assume that $A$ has the maximal $L^p$-regularity for $p\in (1,\infty),$ $B$ is  $p$-admissible for $A,$ and $K=(-A)^{\al}$ for $\al\in (0,\frac{1}{p})$. Then the condition {\bf(H)} is verified, see \cite{AmBoDrHa1}. 
\end{remark}

The first part of the following result was obtained in \cite{HaddManzoRhandi}, while the second part is taken from \cite{AmBoDrHa2}.
\begin{theorem}\label{2015-2021}
  Let the condition {\bf(H)} be satisfied. Then $Z\subset D(C_{\Lambda})$, $C_{\Lambda} = K$ on $Z$ and the operator $\A$ defined by \eqref{closed} generates a strongly continuous semigroup $\T:=(\T(t))_{t\ge 0}$ on $X$ given by
  \begin{equation}\label{WS-VCF}
  	\T(t)x_0 = T(t)x_0 + \int_0^t T_{-1}(t-s)BC_{\Lambda}\T(s)x_0 ds \qquad \qquad x_0\in X\ ,\ t \geq 0.
  \end{equation}
   Furthermore, if in addition $T$ is an analytic semigroup, then it is so for $\T$.
\end{theorem}

Now we state the following lemma.
\begin{lemma}\label{lemma:admissibility for calA}
	If $(A, B, \Pp)$ generates a $p$-regular linear system, then $\calP$ is a $p$-admissible observation for $\A$ whenever condition \textbf{(H)} is satisfied.
\end{lemma}
\begin{proof}
	Condition \textbf{(H)} asserts that $\A$ is a generator of a $C_0$-semigroup $\T$ given by \eqref{WS-VCF}.
	Now, we first remark that $Z\subset D(\Pp_{\Lambda})$ and $P=\Pp_{\Lambda}$ on $Z,$ where  $\Pp_{\Lambda}$ denotes the Yosida extension of $\Pp$ w.r.t. $A$. Let $x\in D(\A)$ and $\al>0$. The facts that $(A,B,\Pp)$ is $p$-regular and $C_\Lambda$ is a $p$-admissible observation for $\T$ imply respectively that
	\begin{align}\label{hhhhh}
		&\int^t_0 T_{-1}(t-s)BC_\Lambda \T(s)x\in D(\Pp_{\Lambda})\quad\text{a.e.}\;t\ge 0,\;\text{and}\cr & \left\| \Pp_{\Lambda}\int^{\cdot}_0 T_{-1}(t-s)BC_\Lambda \T(s)x\right\|_{L^p([0,\al],X)} \le \beta_\al \|x\|,
	\end{align}
	where $\beta_\al>0$ is a constant. On the other hand, by \eqref{WS-VCF}, we have
	\begin{align*}
		\calP\T(t)x&=\Pp_{\Lambda}\T(t)x,\cr &=\Pp_{\Lambda}T(t)x+\Pp_{\Lambda}\int^{t}_0 T_{-1}(t-s)BC_\Lambda \T(s)x.
	\end{align*}
	Hence $p$-admissibility of $\calP$ for $\A$ follows by \eqref{hhhhh} and  $p$-admissibility of $\Pp$ for $A$.
\end{proof}

The following theorem is a result of  maximal $L^p$-regularity for the operator $\A$ in the case of bounded perturbation $K$, see Remak \ref{remark-about-H} (i), Theorem \ref{2015-2021}, and \cite{AmBoDrHa1}.

\begin{theorem}\label{theorem:Desch-Schappacher for MR}
	If $B$ is a $p$-admissible control operator for $A$ and $K$ is bounded, then if  $A$ has maximal $L^p$-regularity, then so has $\A$.
\end{theorem}

The following result (see \cite[Thm.4]{AmBoDrHa2}) gives conditions implying the maximal $L^p$-regularity for $\A$ when the state space $X$ is a UMD space. The result is based on the concept of $\calR$-boundedness.

\begin{theorem}\label{theorem:Staffans-Weiss for MR}
	Let $X$ be an UMD space, $p,p'\in (1,\infty)$ such that $\frac{1}{p} + \frac{1}{p'} = 1$ and the condition {\textbf{(H)}} be satisfied. Assume that there exists $\om>\max\{\om_0(A), \om_0(\A)\}$ such that the sets $\{s^\frac{1}{p} R(\om+is, A_{-1})B, s\neq 0 \}$ and $\{s^\frac{1}{p'} CR(\om+is, A), s\neq 0 \}$ are $\calR$-bounded. If $A$ has maximal $L^p$-regularity, then $\A$ has the same property.
\end{theorem}

Let us now prove the maximal $L^p$-regularity for the boundary Volterra integro-differential equations \eqref{InDiffwBC}. We start with the following special case of bounded boundary perturbation $K$.
\begin{theorem}\label{theorem:MR for Volterra with Desch-Schappacher}
	Let $X$ be a UMD space and $q\in (1,\infty)$. Assume that $a(\cdot)\in B_{\theta}^q$ for some $\theta\in (0, \displaystyle\frac{\pi}{2}]$, $B$ is an $l_0$-admissible control operator for $A$ for some $l_0\in (1,\infty)$, $K\in \calL(X, U)$ and $(A, B, \Pp)$ generates an $l_0$-regular linear system. If both $A$ and $\displaystyle\frac{d}{dz}$ have maximal $L^p$-regularity on $X$ and $B^q_{\theta, X}$ respectively for some $p\in (1, \infty)$, then $\mathfrak{A}$ has maximal $L^p$-regularity on $\mathfrak{X}^q$. Moreover, if $p\in (1, l_0]$ then there exists a constant $C_p>0$ independent of $f\in L^p([0,T],X)$ such that the solution $z$ of \eqref{InDiffwBC} satisfies
	\begin{equation}
	\|\dot{z}\|_{L^p([0,T],X)} + \|A_m z\|_{L^p([0,T],X)} + \|z\|_{L^p([0,T],X)}\leq C_p \|f\|_{L^p([0,T],X)}.
	\end{equation}
\end{theorem}
\begin{proof}
	By Theorem \ref{theorem:Desch-Schappacher for MR}, $l_0$-admissibility of $B$ and maximal regularity of $A$ imply that $\A$ has maximal $L^p$-regularity. By Lemma \ref{lemma:admissibility for calA}, the $l_0$-regularity of the system generated by $(A, B, \Pp)$ asserts that $\calP$ is an $l_0$-admissible observation operator for $\A$. These facts together with maximal $L^p$-regularity of $\displaystyle\frac{d}{dz}$ imply, by Theorem \ref{good-thm}, that $\mathfrak{A}$ has maximal $L^p$-regularity. The estimate follows by the same theorem. This ends the proof.
\end{proof}

Now, we suppose that $K$ is unbounded and state the second main theorem of this paper.

\begin{theorem}\label{theorem:MR for Volterra with Staffans-Weiss}
	Let $X$ be a UMD space and $q\in (1,\infty)$. Assume that $a(\cdot)\in B_{\theta}^q$ for some $\theta\in (0, \displaystyle\frac{\pi}{2}]$, $(A, B, C)$ and $(A, B, \Pp)$ generate $l_0$-regular linear systems for some $l_0\in (1,\infty)$ and the identity $I_U$ is an admissible feedback for the system generated by $(A, B, C)$. If both $A$ and $\displaystyle\frac{d}{dz}$ have maximal $L^p$-regularity on $X$ and $B^q_{\theta, X}$ respectively for some $p\in (1, \infty)$ and the sets $\{s^\frac{1}{p} R(is, A_{-1})B, s\neq 0 \}$ and $\{s^\frac{1}{p'} CR(is, A), s\neq 0 \}$ are $\calR$-bounded, then $\mathfrak{A}$ has maximal $L^p$-regularity on $\mathfrak{X}^q$. Moreover, if $p\in (1, l_0]$ then there exists a constant $C_p>0$ independent of $f\in L^p([0,T],X)$ such that the solution $z$ of \eqref{InDiffwBC} satisfies
	\begin{equation}
	\|\dot{z}\|_{L^p([0,T],X)} + \|A_m z\|_{L^p([0,T],X)} + \|z\|_{L^p([0,T],X)}\leq C_p \|f\|_{L^p([0,T],X)}.
	\end{equation}
\end{theorem}
\begin{proof}
	By Theorem \ref{theorem:Staffans-Weiss for MR}, the assumptions imply that $\A$ has maximal $L^p$-regularity. By Lemma \ref{lemma:admissibility for calA}, $l_0$-regularity of the system generated by $(A, B, \Pp)$ asserts that $\calP$ is an $l_0$-admissible observation operator for $\A$. Gathering these facts with maximal $L^p$-regularity of $\displaystyle\frac{d}{dz}$, we conclude by Theorem \ref{good-thm}, that $\mathfrak{A}$ has maximal $L^p$-regularity. The same theorem justifies the estimate which ends the proof.
\end{proof}


\begin{thebibliography}{99}
\addcontentsline{toc}{section}{Bibliographie}

	\bibitem{AmBoDrHa1}A. Amansag, H. Bounit, A. Driouich, S. Hadd On the maximal regularity for perturbed autonomous and non-autonomous evolution equations. J. Evol. Equ. 20, 165–190 (2020). https://doi.org/10.1007/s00028-019-00514-8
	\bibitem{AmBoDrHa2} A. Amansag, H. Bounit, A. Driouich, S. Hadd, 	Staffans-Weiss perturbations for Maximal $L^p$-regularity in Banach spaces, in press in J. Evol. Equ.
	\bibitem{ArRaFoPo} W. Arendt, R. Chill, S. Fornaro and C. Poupaud, $L^p$-maximal regularity for nonautonomous evolution equations, \emph{ J. Differ. Equations} \textbf{237}, 1--26, (2007).
	\bibitem{Barta1} T. B\'arta, Analytic solutions of Volterra equations via semigroups, \emph{Semigroup Forum} \textbf{76}, 1, 142--148, (2008).
	\bibitem{Barta2} T. B\'arta, On R-sectorial derivatives on Bergman spaces, \emph{ Bull. Austral. Math. Soc.} \textbf{77} , 305--313, (2008).
	\bibitem{Barta3} T. B\'arta, Smooth solutions of  Volterra equations via semigroups, \emph{ Bulletin of the Australian Mathematical Society}, \textbf{78}(2), (2008), 249--260. doi:10.1017/S0004972708000683
	\bibitem{Simon} L.De Simon,\emph{ Un’ applicazione della theoria degli integrali singolari allo studio delle equazioni differenziali lineare astratte del primo ordine},Rend. Sem. Mat., Univ. Padova \textbf{99} (1964), 205-223.
	\bibitem{DenHiePru}  R. Denk, M. Hieber, and J. Pruss, R-Boundedness, \emph{Fourier Multipliers and Problems of Elliptic and Parabolic Type}, Memoirs Amer. Math. Soc., vol. 166, Amer. Math. Soc., Providence, R.I., 2003.
	\bibitem{Dore} G. Dore, \emph{Maximal regularity in Lp spaces for an abstract Cauchy problem}. Adv. Differ. Equat. 5(1-3), 293-322 (2000).
	\bibitem{EngNag} K.-J. Engel and R. Nagel. \emph{One-Parameter Semigroups for Linear Evolution Equations}, Springer-Verlag, New York,Berlin,Heidelberg,2000.
	\bibitem{Grei} G. Greiner. \emph{Perturbing the boundary conditions of a generator}. Houston J. Math., 18:405-425, 2001.
	\bibitem{Hadd} S. Hadd. \emph{Unbounded perturbations of $C_0$-semigroups on Banach spaces and applications}. Semigroup Forum, 70:451-465, 2005.
	\bibitem{HaddManzoRhandi} S. Hadd, R. Manzo, A. Rhandi.\emph{ Unbounded perturbations of the generator domain}. Discrete \& Continuous Dynamical Systems - A, 2015, 35 (2) : 703-723. doi: 10.3934/dcds.2015.35.703
	\bibitem{HieMon} M. Hieber, S. Monniaux, Pseudo-differential operators and maximal regularity results for non-autonomous parabolic equations, \emph{ Proc. Amer. Math. Soc.} \textbf{128} 1047--1053, (2000).
	\bibitem{KW01} N.J. Kalton and L. Weis, The $H^\infty$ -calculus and sums of closed operators, \emph{ Math. Ann.} \textbf{321},319--345, (2001).
	\bibitem{KunstmannWeis}P. C. Kunstmann and L. Weis, \emph{Maximal $L^p$ regularity for parabolic equations, Fourier multiplier theorems and $H^\infty$ functional calculus}, Levico Lectures, Proceedings of the Autumn School on Evolution Equations and Semigroups (M. Iannelli, R. Nagel, S.Piazzera eds.), vol. 69, Springer Verlag, Heidelberg, Berlin, 2004, pp. 65-320
	\bibitem{KunWei} P. C. Kunstmann and L. Weis,\emph{ Perturbation theorems for maximal Lp-regularity},Ann. Scuola Norm. Sup. Pisa Cl. Sci. (4) 30 (2001), 415-435.
	\bibitem{Lions61} J.-L. Lions, Equations diff$\acute{e}$rentielles op$\acute{e}$rationnelles et probl$\grave{e}$mes aux limites, \emph{ Die Grundlehren der mathematischen Wissenschaften}, Bd. 111, Springer-Verlag, Berlin, (1961).
	\bibitem{LeMerdy} C. Le Merdy. The Weiss conjecture for bounded analytic semigroups. J. Lond. Math. Soc. 67(3):715–738, 2003.
	\bibitem{Pruss} J. Pr\"uss, Evolutionary Integral Equations and Applications, Monographs in Mathematics 87, Birkh\"auser, Basel, (1993).
	\bibitem{LutzWeis} L. Weis, \emph{A new approach to maximal L p -regularity}, Proc. 6th International Conference on Evolution Equations, G. Lumer and L. Weis, eds, Dekker, New York (2000), 195-214.
	\bibitem{LutzWeis2} L. Weis,\emph{ Operator-valued Fourier multiplier theorems and maximal Lp-regularity}, Math. Ann.,319 (2001), 735-758.
	\bibitem{Wood07}
	\newblock I. Wood, Maximal $L^p$-regularity for the Laplacian on Lipschitz domain, \emph{ Math.Z}, \textbf{255} , pp. 855--875, (2007).
	\bibitem{Zacher}  R. Zacher, Maximal regularity of type L p for abstract parabolic Volterra equations, \emph{J. evol. equ.} \textbf{5} , pp:79--103, (2005).

\end{thebibliography}
\end{document}